\documentclass[twoside,12pt]{article}
\usepackage{color}

\usepackage{graphics}
\usepackage{graphicx}

\usepackage{amssymb,amsfonts}
\usepackage{amsmath}
\usepackage{amsthm}

\def\hpp{${\cal HPP}$\ }
\renewcommand{\vec}[1]{\mathbf{#1}}
\def\veca{{\boldsymbol{\alpha}}}

\def\binomh#1#2{ \scalebox{.3}[1.2]{\textbf{)}}{\genfrac{}{}{0pt}{}{#1}{#2}}\scalebox{.3}[1.2]{\textbf{(}} }

\newtheorem{theorem}{Theorem}

\newtheorem{rem}{Remark}

\title{\bf Hyperbolic Pascal pyramid 
}
\author{ L\'aszl\'o N\'emeth\footnote{University of West Hungary,  Institute of Mathematics, Hungary. \textit{nemeth.laszlo@emk.nyme.hu}}}
\date{}

\begin{document}

\maketitle

\begin{abstract}
In this paper we introduce a new type of Pascal's pyramids. The new object is called hyperbolic Pascal pyramid since the mathematical background goes back to the regular cube mosaic (cubic honeycomb) in the hyperbolic space. The definition of the hyperbolic Pascal pyramid is a natural generalization of the definition of hyperbolic Pascal triangle (\cite{BNSz}) and Pascal's arithmetic pyramid. We describe the growing of hyperbolic Pascal pyramid considering the numbers and the values of the elements. Further figures illustrate the stepping from a level to the next one.   \\[1mm]
\end{abstract}

\section{Introduction}\label{sec:introduction} 

There are several approaches to generalize the Pascal's arithmetic triangle (see, for instance \cite{BSz}). A new type of variations of it is based on the hyperbolic regular mosaics denoted by Schl\"afli's symbol $\{p,q\}$, where $(p-2)(q-2)>4$ (\cite{C}). Each regular mosaic induces a so called hyperbolic Pascal triangle (see \cite{BNSz, NSz1}), following and generalizing the connection between the classical Pascal's triangle and the Euclidean regular square mosaic $\{4,4\}$. For more details see \cite{BNSz}, but here we also collect some necessary information. 

The hyperbolic Pascal triangle based on the mosaic $\{p,q\}$ can be figured as a digraph, where the vertices and the edges are the vertices and the edges of a well defined part of  lattice $\{p,q\}$, respectively, and the vertices possess a value that give the number of different shortest paths from the base vertex to the given vertex. Figure~\ref{fig:Pascal_layer6} illustrates the hyperbolic Pascal triangle when $\{p,q\}=\{4,5\}$. 
Here the base vertex has two edges, the leftmost and the rightmost vertices have three, the others have five edges. The quadrilateral shape cells surrounded by the appropriate edges correspond to the squares in the mosaic.
Apart from the winger elements, certain vertices (called ``Type $A$'') have $2$ ascendants and $3$ descendants, while the others (``Type $B$'') have $1$ ascendant and $4$ descendants. In the figures we denote  vertices type $A$ by red circles and  vertices type $B$ by cyan diamonds, further the wingers by white diamonds (according to the denotations in \cite{BNSz}). The vertices which are $n$-edge-long far from the base vertex are in row $n$. 
The general method of preparing the graph is the following: we go along the vertices of the $j^{\text{th}}$ row, according to the type of the elements (winger, $A$, $B$), we draw the appropriate number of edges downwards ($2$, $3$, $4$, respectively). Neighbour edges of two neighbour vertices of the $j^{\text{th}}$ row meet in the $(j+1)^{\text{th}}$ row, constructing a new vertex type $A$. The other descendants of row $j$ have type $B$ in row $j+1$.
In the sequel, $\binomh{n}{k}$ denotes the $k^\text{th}$ element in row $n$, which is either the sum of the values of its two ascendants or the value of its unique ascendant. We note, that the hyperbolic Pascal triangle has the property of vertical symmetry. 

\begin{figure}[h!]
 \centering
  \includegraphics[width=0.99\linewidth]{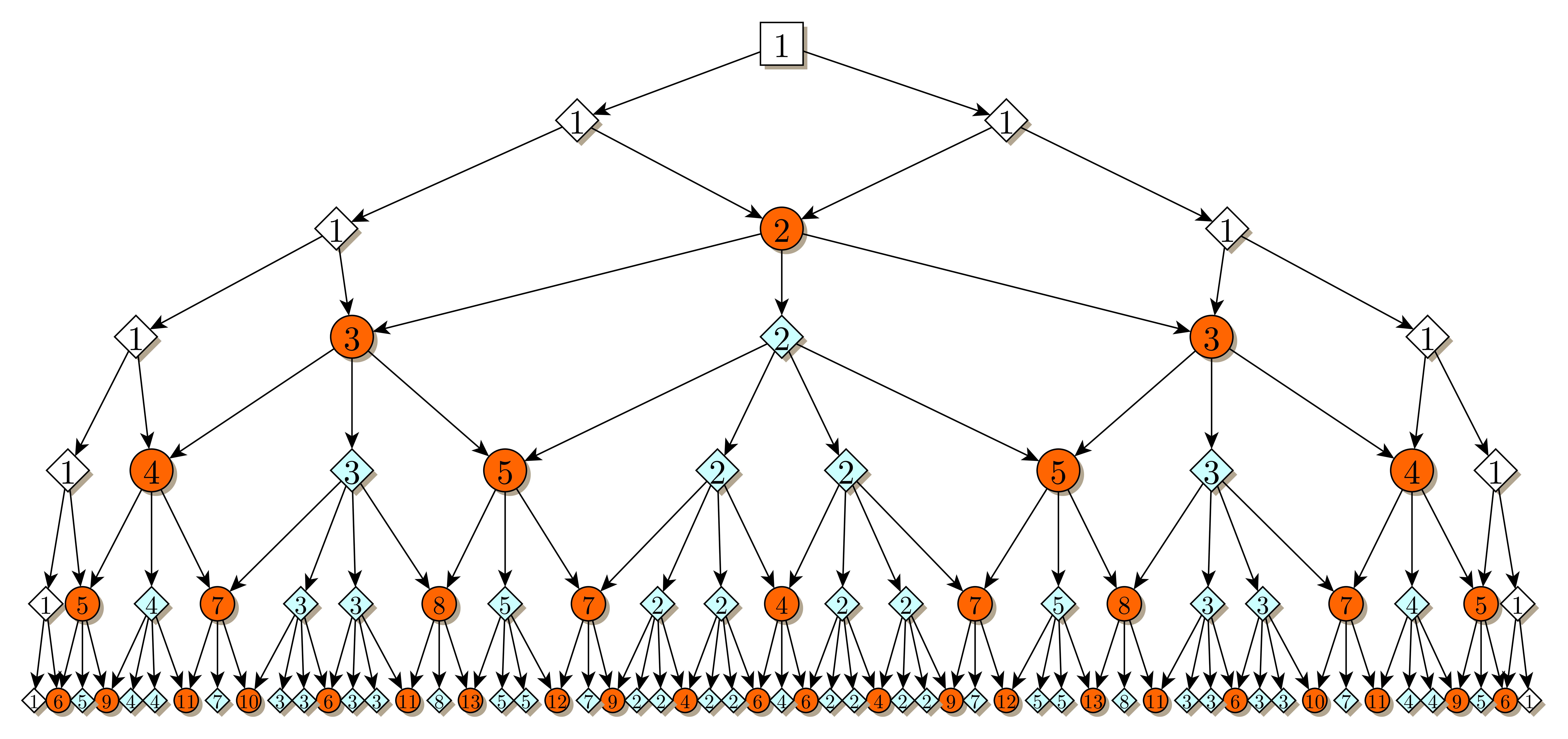}
 \caption{Hyperbolic Pascal triangle linked to $\{4,5\}$ up to row 6}
 \label{fig:Pascal_layer6}
\end{figure}

The 3-dimensional analogue of the original Pascal's triangle is the well-known Pascal's pyramid (or more precisely Pascal's tetrahedron) (left part in Figure~\ref{fig:Eulidean_pyramid}). 
Its levels are triangles and the numbers along the three edges of the $n^{\text{th}}$ level are the numbers of the $n^{\text{th}}$ line of Pascal's triangle. 
Each number inside in any levels is the sum of the three adjacent numbers on the level above \cite{B, har}.

In the following we define a Pascal pyramid in the hyperbolic space based on the hyperbolic regular cube mosaic (cubic honeycomb) with Schl\"afli's symbol $\{4,3,5\}$ as a generalisation of the hyperbolic Pascal triangle and the classical Pascal's pyramid which are based on the hyperbolic planar mosaic $\{4,5\}$ and  the Euclidean regular cube mosaic $\{4,3,4\}$, respectively. (We write the hyperbolic  one without an ``apostrophe", similarly to the writing of the classical Pascal's triangle and the hyperbolic Pascal triangle.)

\section{Construction of the hyperbolic Pascal pyramid}

In the hyperbolic space there are 7 regular mosaics and one of them is the regular cube mosaic $\{4,3,5\}$ (see \cite{C}), which is the hyperbolic analogue of the Euclidean regular cube mosaic.

First of all for the further examination we summarise some properties of the cubic honeycomb $\{4,3,5\}$, which is not as well-known as the Euclidean one. 
The vertex figures of the hyperbolic cube mosaic are icosahedra with Schl\"afli's symbol $\{3,5\}$. Thus the nearest vertices to a certain vertex $V$ form an icosahedron in this mosaic.  It means, considering an arbitrary vertex $V$ of the mosaic, that the number of  cubes around $V$ is as many as the number of the faces of the icosahedron, namely 20 and the number of the mosaic edges from $V$ (degree of $V$) is as many as the number of vertices on an icosahedron, namely 12. There are 5 cubes around a mosaic edge as there are 5 faces around a vertex on the icosahedron. In Figure~\ref{fig:belt0} we can see a vertex figure with one cube and the twenty cubes around the centre $V$ of the icosahedron. Vertex $X$ and $W$ are two nearest vertices of the mosaic to $V$ and around the edge $V$-$W$ there are 5 cubes.
We mention that the edges of the icosahedron are not the edges of the mosaic, they are the face diagonals of the cubes.

\begin{figure}[h!]
 \centering
 \includegraphics[scale=0.98]{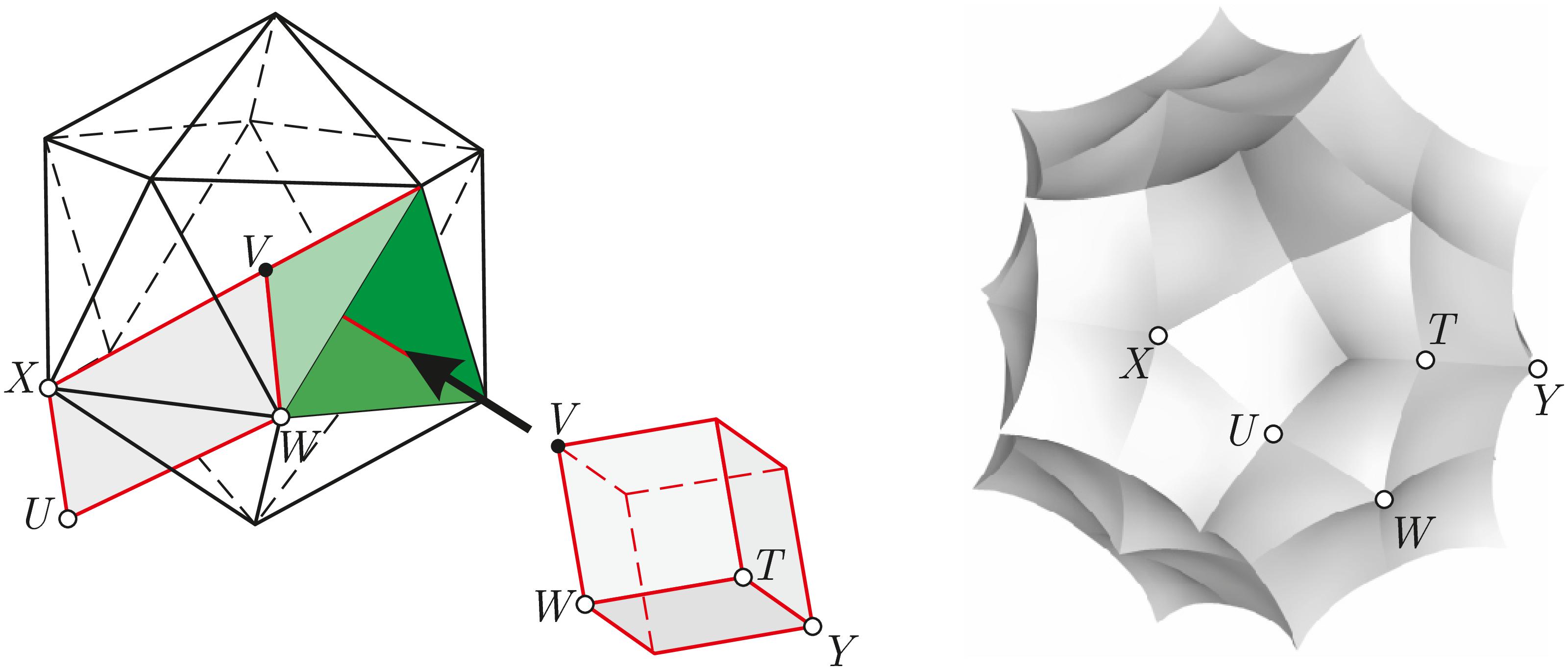}
  \caption{Vertex figure of $V$ and cubes around $V$}
 \label{fig:belt0}
\end{figure}

Now, we consider the hyperbolic (and Euclidean) cubic honeycomb. We define the part ${\cal P}$ of the mosaic which can induce the hyperbolic Pascal pyramid (and the classical  Pascal's pyramid).

Take a cube of the mosaic as a base cell of ${\cal P}$ and let $V_0$ be a vertex of it. Take the three cubes of the mosaic which have common faces with the base cell but do not contain $V_0$. (In the right part of Figure~\ref{fig:border3d} we can see the construction of ${\cal P}$ and by comparison the left part shows the Euclidean one.) Reflect these cubes across their own faces which are opposite the touching faces with the base cube. Reflect again the new cubes across the faces which are opposite the previous cubes, and so on limitless. This way we give the "edge"s of the border of ${\cal P}$ (blue cubes in Figure~\ref{fig:border3d}) and the convex parts of the mosaic-levels defined by any two ``edge"s give the border of ${\cal P}$. Finally, the convex part of the bordered parts of the mosaic is the well defined ${\cal P}$. The shape of this convex part of the mosaic resembles an infinite tetrahedron. 
                                                                  
Let ${\cal G}_{\cal P}$ be the graph, in which the vertices and edges are the vertices and edges of ${\cal P}$. We label an arbitrary vertex $V$ of ${\cal G}_{\cal P}$  by the number of different shortest paths along the edges of ${\cal P}$ from $V_0$ to $V$. We mention that all the edges of the mosaic are equivalent. Some labelled vertices can be seen in Figure \ref{fig:border3d}.
Let the labelled ${\cal G}_{\cal P}$ be the hyperbolic Pascal pyramid (more precisely the hyperbolic Pascal tetrahedron), denoted by ${\cal HPP}$.  Considering the Euclidean mosaic  $\{4,3,4\}$ instead of the hyperbolic one in the definition above the classical Pascal's pyramid returns.
  
\begin{figure}[h!]
 \centering
 \includegraphics{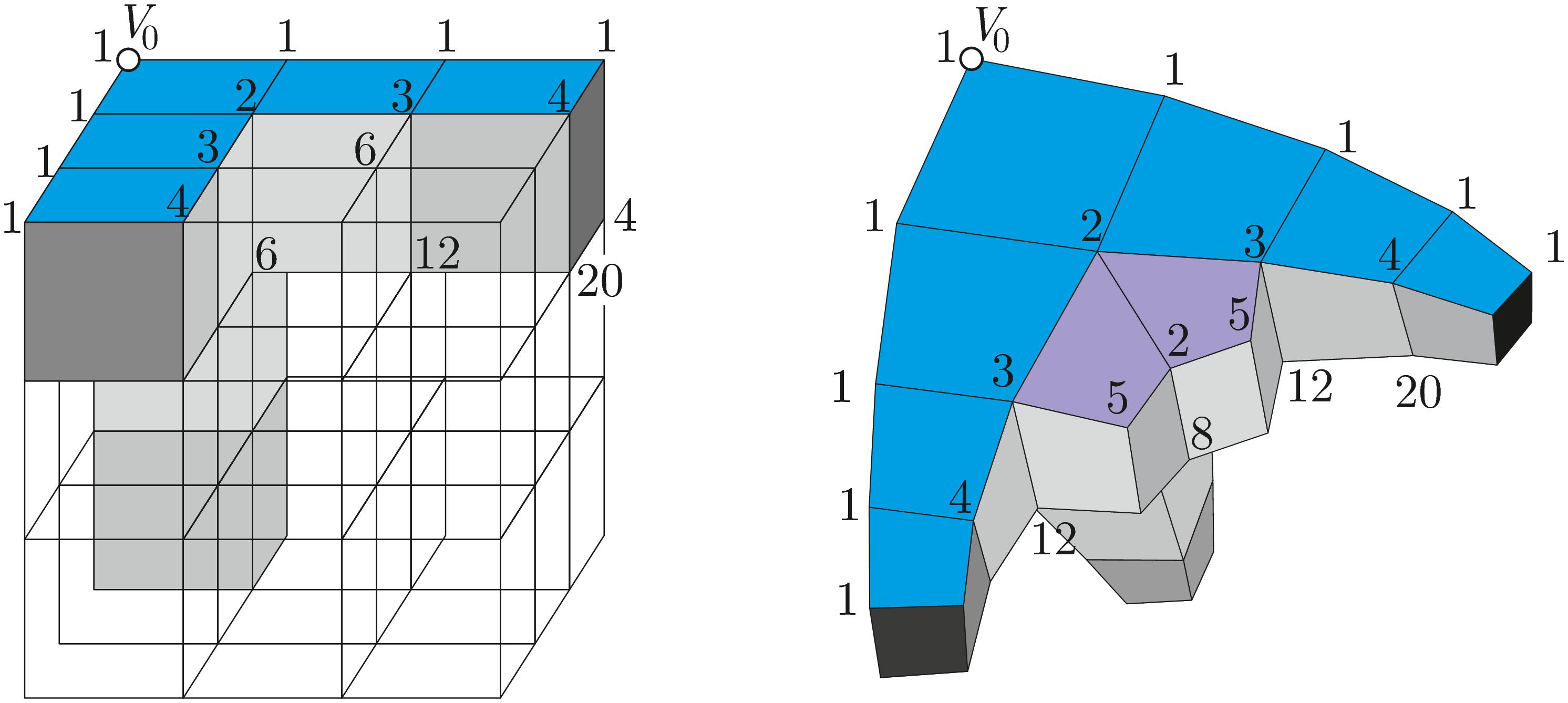}
 \caption{Construction of the border of $\cal{P}$}
 \label{fig:border3d}
\end{figure}

Let level~0 be the vertex $V_0$. Level~$n$ consists of the vertices of \hpp whose edge-distances from $V_0$ are $n$-edge (the distance of the shortest path along the edges of ${\cal P}$ is $n$).   
It is clear, that the labelled graphs indicated by  the outer boundaries of ${\cal P}$ are the hyperbolic Pascal triangles based on the regular hyperbolic planar mosaic $\{4,5\}$. 

The right part of Figure~\ref{fig:hyperbolic_pyramid} shows the hyperbolic Pascal pyramid up to level~4, when the digraph ${\cal G}_{\cal P}$ is directed from $V_0$ according to the growing  distance from $V_0$ (compare Figures \ref{fig:border3d} and \ref{fig:hyperbolic_pyramid}). 
Moreover, Figures~\ref{fig:layer3to4} and \ref{fig:layer4to5} show the growing from a level to the next one in case of some lower levels. The colours and shapes of different types of the vertices are different. (See the definitions later.) The numbers without colouring and shapes refer to vertices in the lower level in every figure. The graphs growing from a level to the new one contain graph-cycle with six nodes. These graph-cycles figure the convex hulls of the parallel projections of the cubes from the mosaic, where the direction of the projection is not parallel to any edges of the cubes.   
                                                                             
\begin{figure}[h!]
 \centering
 \includegraphics[width=0.48\textwidth]{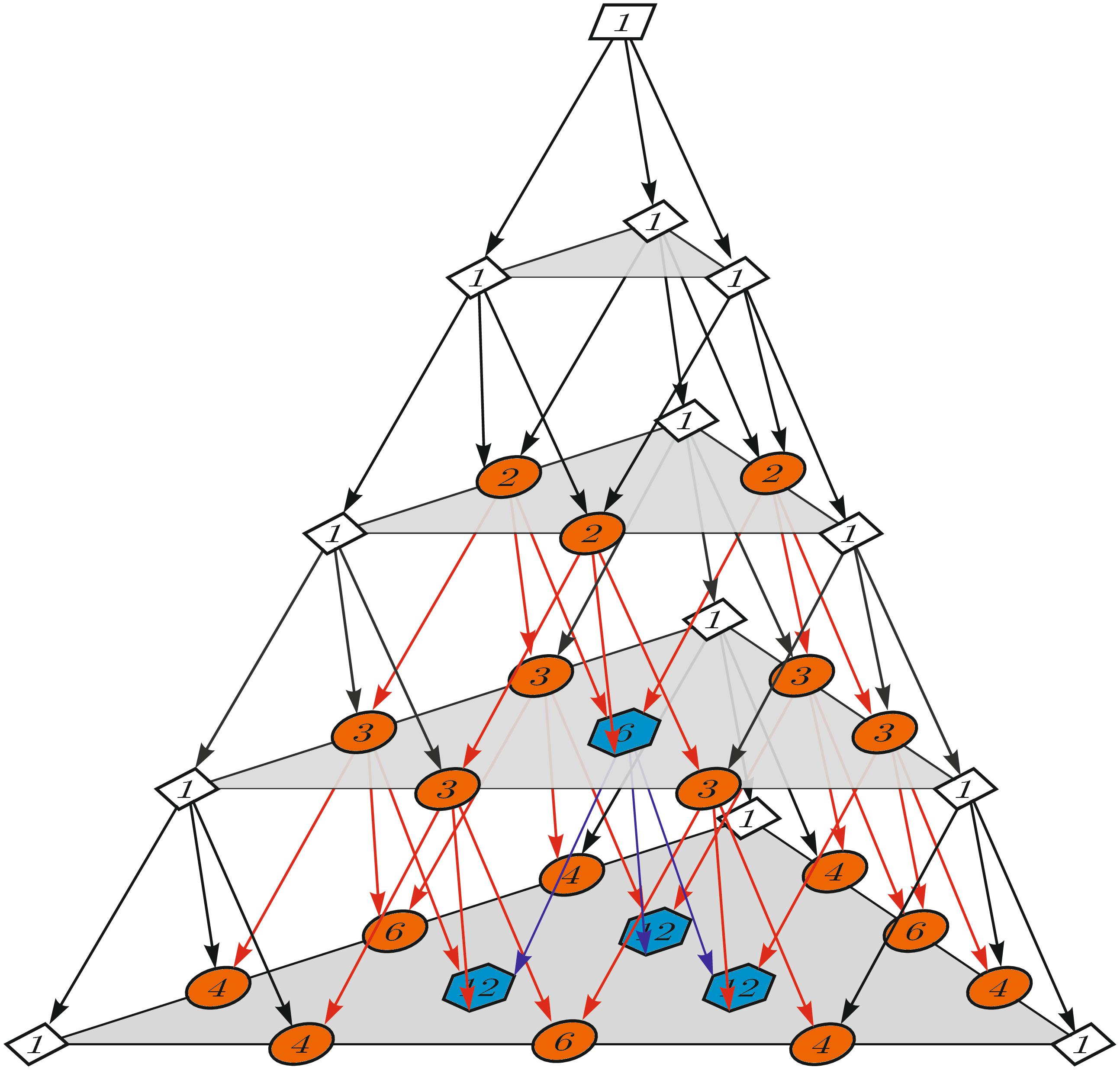}  \includegraphics[width=0.48\textwidth]{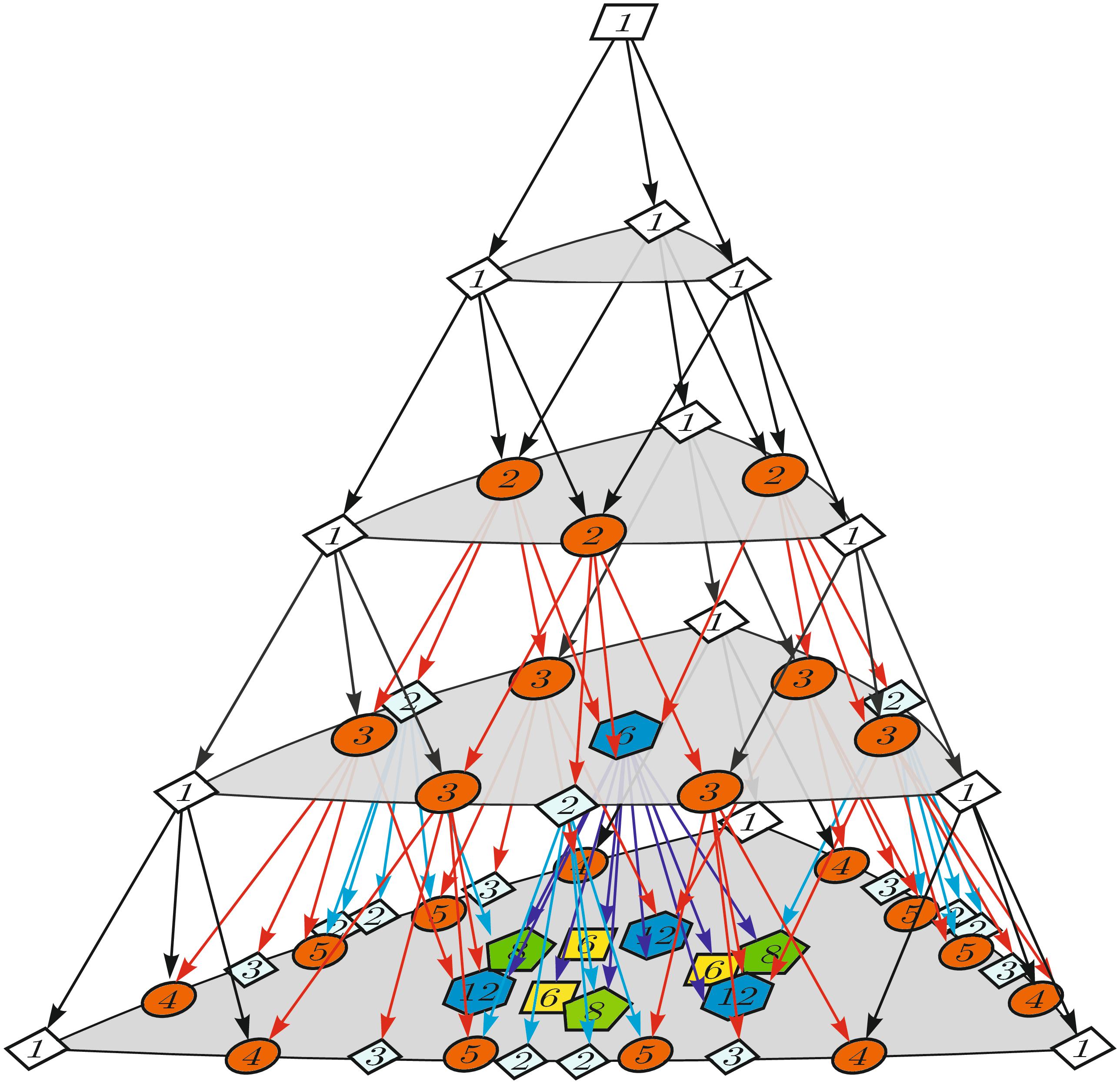}
 \caption{Euclidean and hyperbolic Pascal pyramid}
 \label{fig:hyperbolic_pyramid}
 \label{fig:Eulidean_pyramid}
\end{figure}

\begin{figure}[h!]
 \centering
 \includegraphics[scale=0.8]{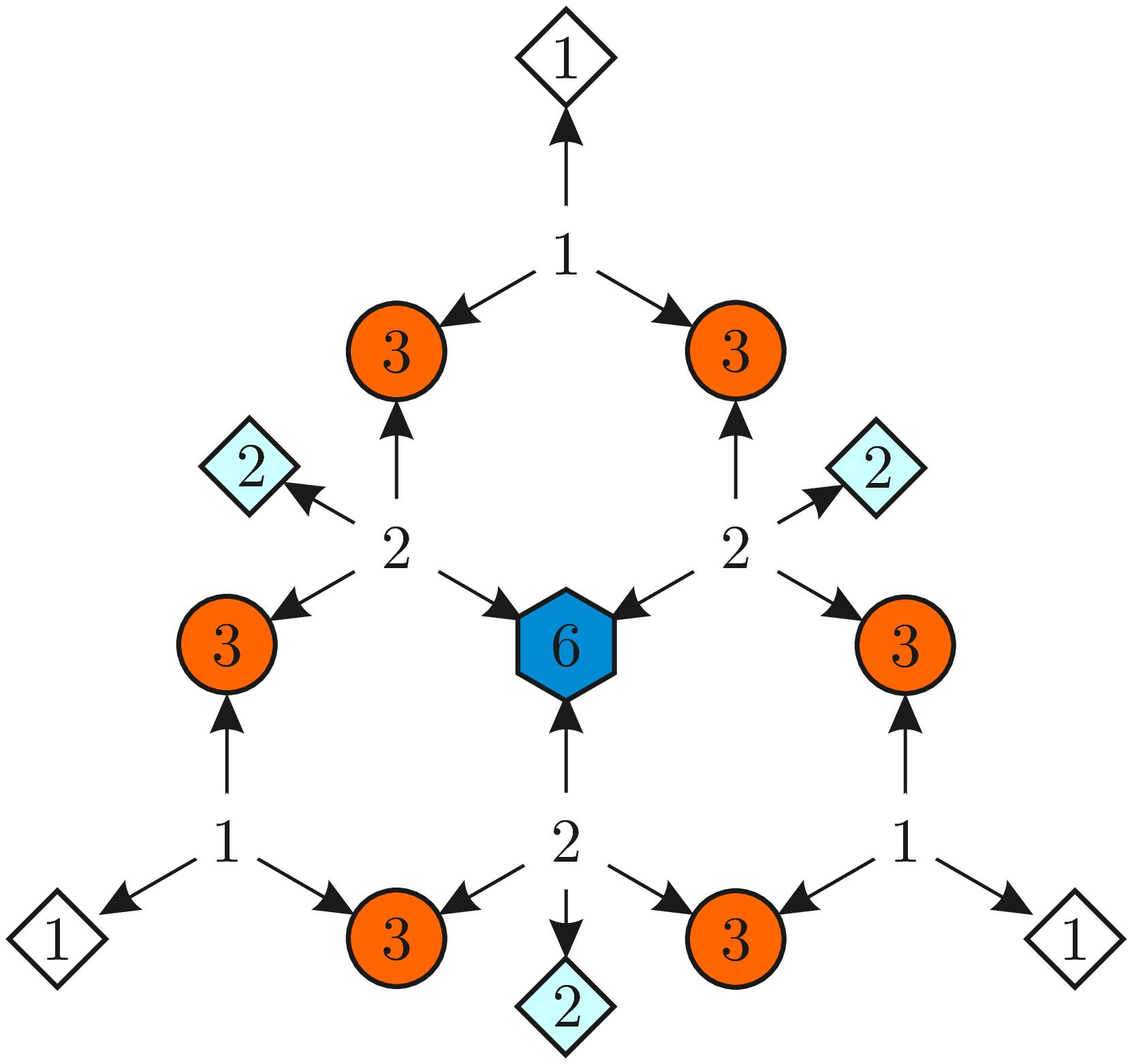}
 \includegraphics[scale=0.8]{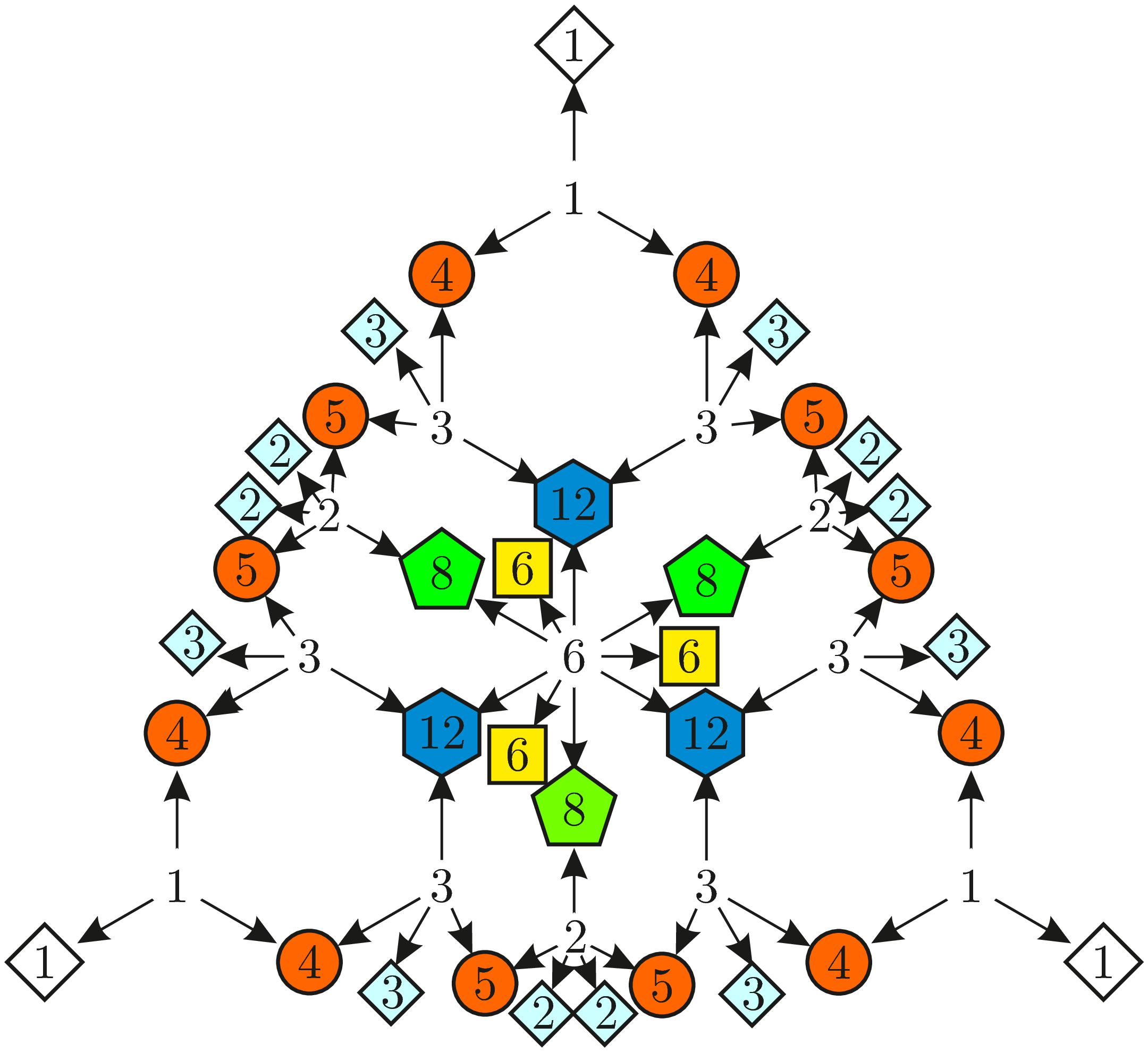}
 \caption{Connection between levels two, three and four in \hpp}
 \label{fig:layer3to4}
\end{figure}

\begin{figure}[ht!]
 \centering
 \includegraphics[width=0.9\linewidth]{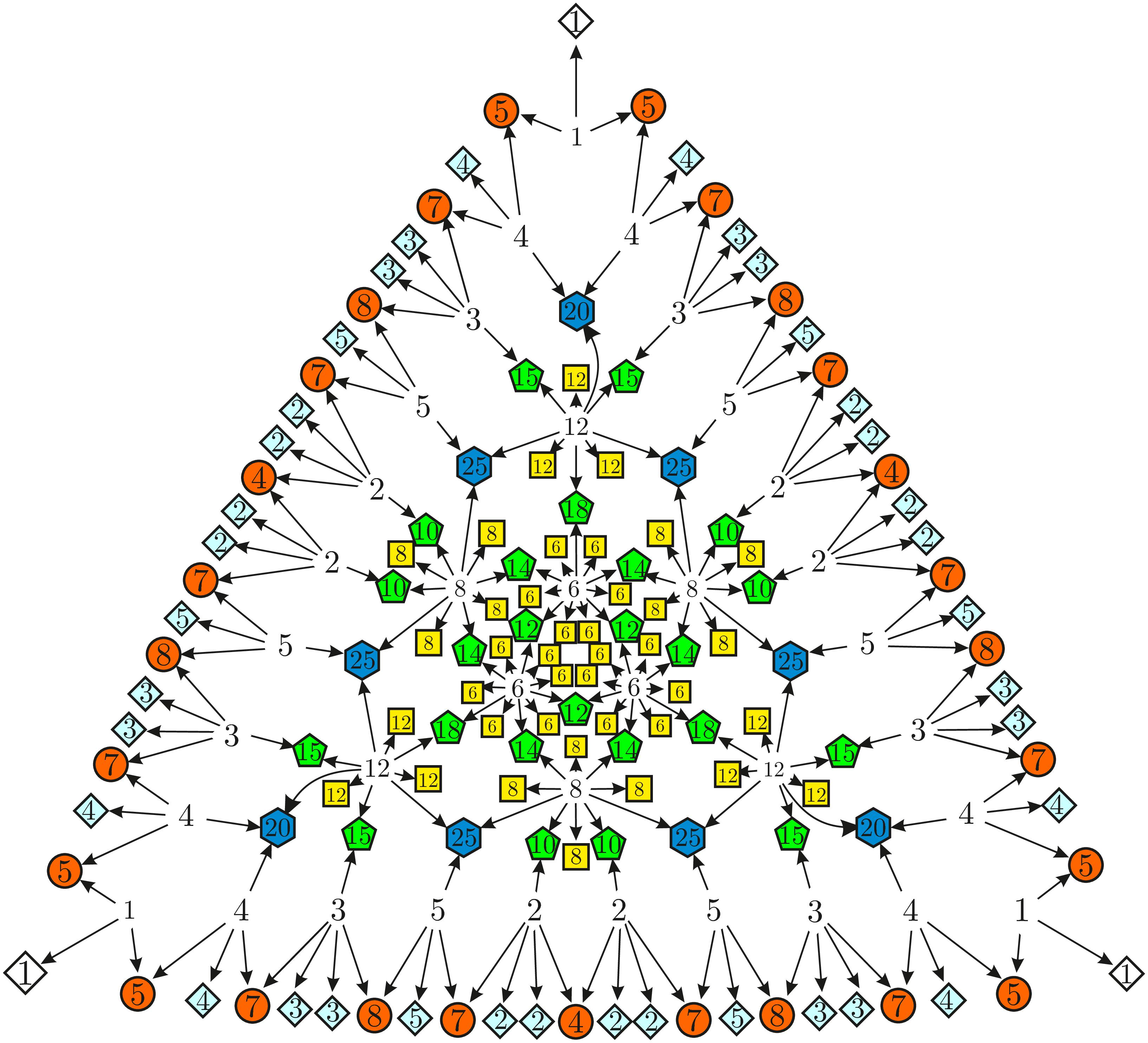}
 \caption{Connection between levels four and five in \hpp}
 \label{fig:layer4to5}
\end{figure}

In the following we describe the method of the growing of the hyperbolic Pascal pyramid and we give the sum of the paths connecting vertex $V_0$ and level $n$.

\section{Growing of the hyperbolic Pascal pyramid}

In the classical Pascal's pyramid the number of the elements on  level $n$ is $(n+1)(n+2)/2$ and its growing  from level $n$ to $n+1$ is $n+2$, on the contrary in the hyperbolic Pascal pyramid it is more complex. 

As the faces of \hpp are the hyperbolic Pascal triangles, then here are three types of vertices $A$, $B$ and $1$ corresponding to the Introduction and \cite{BNSz}. From all $A$ and $B$ start only one edge to the inside of the pyramid, because five cubes close around an edge of the mosaic (see Figure \ref{fig:border3d}). The types of inside vertices of these edges differ from the types $A$ and $B$, denote them by type $C$ and type $D$, respectively. 
The left part of Figure~\ref{fig:from_border} presents a cube, in which the upper face is on the border of \hpp and a vertex $A$ on level $i$ generates a vertex $C$ inside of \hpp with 3 incoming edges. The right part shows that all vertices $B$ imply a vertex $D$ with two incoming edges.

\begin{figure}[h!]
 \centering
 \includegraphics[scale=0.85]{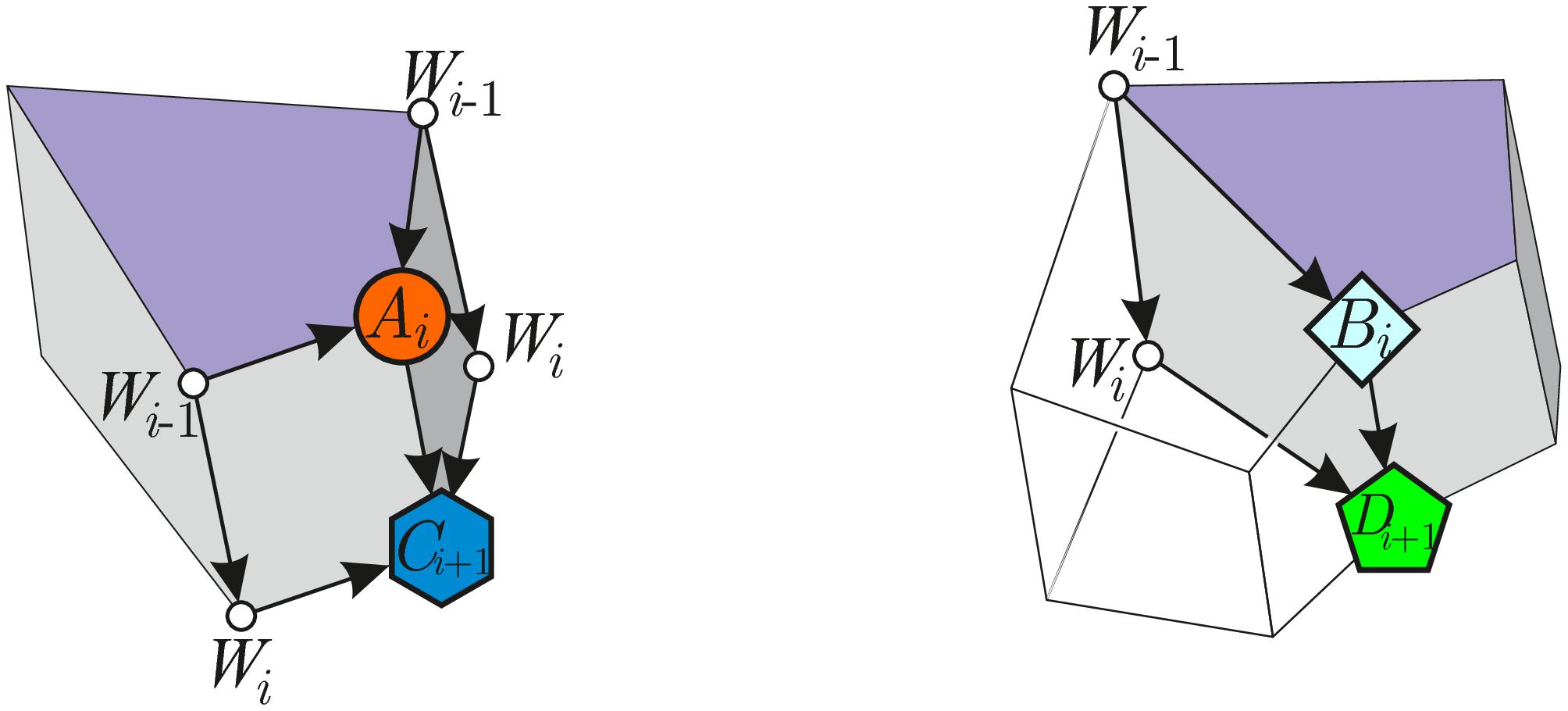}
 \caption{Groing from border to inside}
 \label{fig:from_border}
\end{figure}

The growing methods of them are illustrated in Figure~\ref{fig:gowing3d_1} (compare it with the growing method in \cite{BNSz}).

\begin{figure}[h!]
 \centering
 \includegraphics[scale=0.8]{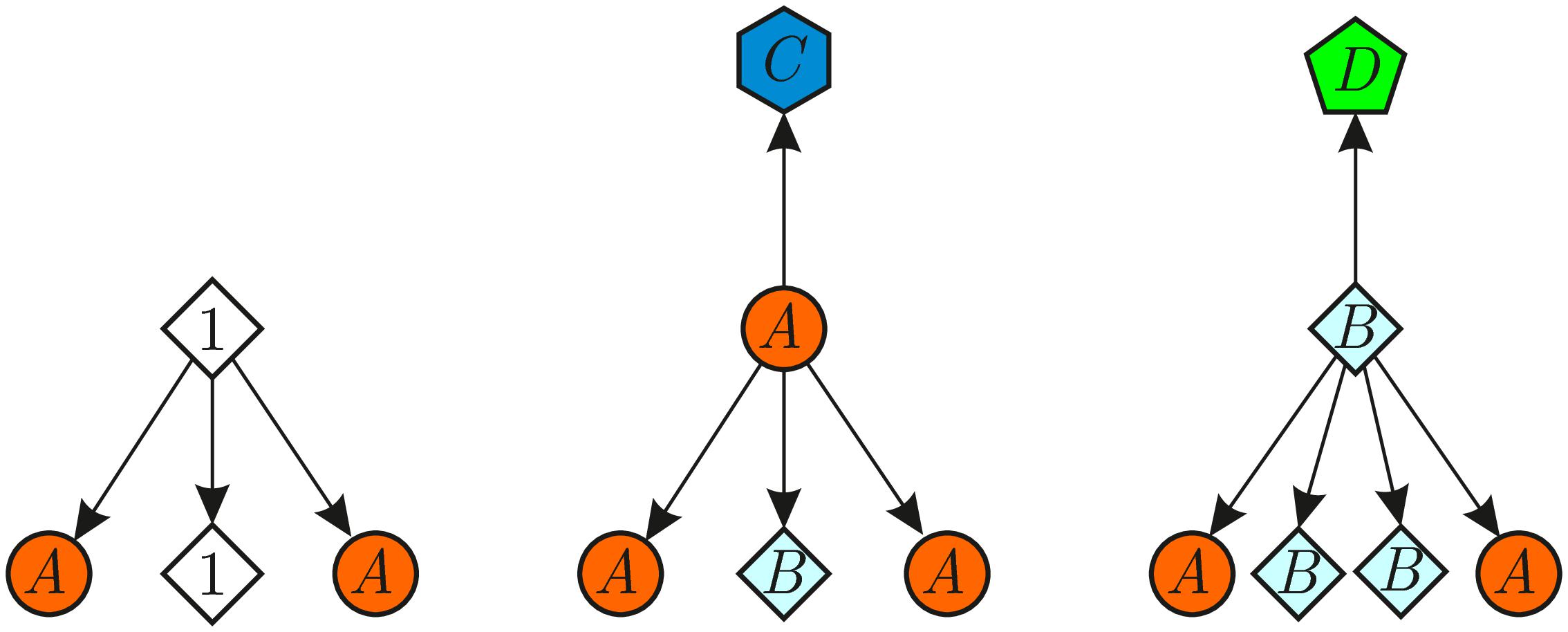}
  \caption{Growing method in case of the faces}
 \label{fig:gowing3d_1}
\end{figure}

As a cube has three edges in all vertices, then during the growing (step from level $i-1$ to level $i$) an inner arbitrary vertex $V$ on level $i$ can be reached from level $i-1$ with three, two or just one edges. This fact allows us a classification of the inner vertices. Let the type of a vertex on level $i$ is $C$, $D$ or $E$, respectively, if it has three, two or one joining edges to level $i-1$ (as before). Figure~\ref{fig:gowing3d_icosa} shows vertex figures of the inner vertices of \hpp\!\!. Vertices $W_{i-1}$ (small green circles) are on  level $i-1$, we don't know their types (or not important to know) and the centres are on  level~$i$. 
The other vertices of the icosahedron are on level $i+1$ and the classification of them gives their types. An edge of the icosahedron and its centre $V$ determine a square (a face of a cube) from the mosaic. (Recall, that an edge of the icosahedron is a diagonal of a face of a cube from the mosaic.)  Since from a vertex of a square we can go to the opposite vertex two ways, then a vertex $X$ of the icosahedron, where $X$ and a $W_{i-1}$ are connected by an edge, can be reached with two paths from level $i-1$. (For example in Figure~\ref{fig:belt0}, between vertex $W$ and $X$ there are the paths $W\!-\!V\!-\!X$ and $W\!-\!U\!-\!X$.)  So, the type of the third vertex of the faces on the icosahedron whose other two vertices are $W_{i-1}$ are $C$. The types of the vertices which connect to only one $W_{i-1}$ with  icosahedron-edge are $D$, the others are type $C$. See Figure~\ref{fig:gowing3d_icosa}. In case of the vertex figure of $C$ or $D$ a vertex $W_{i-1}$ can be  vertex $A$ or $B$, respectively.   
In the figures we denote  vertices type $C$ by blue hexagons,  vertices type $D$ by green pentagons and  vertices type $E$ by yellow squares. The blue thick directed edges are mosaic-edges between levels $i-1$ and $i$, while the red thin ones are between levels $i$ and $i+1$. We mention that in case of Pascal's pyramid there are only type $C$ inner vertices.

\begin{figure}[h!]
 \centering
 \includegraphics[width=0.99\linewidth]{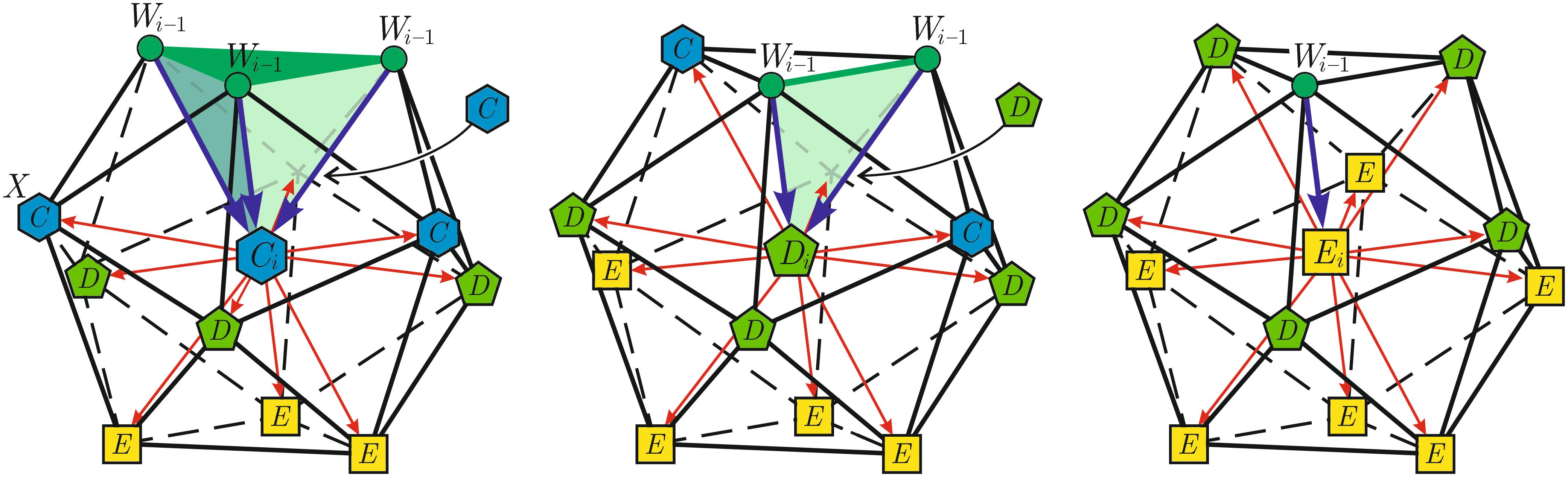}
 \caption{Growing method in case of the inner vertices with icosahedra}
 \label{fig:gowing3d_icosa}
\end{figure}

In Figure~\ref{fig:gowing3d_2} the growing method is presented in case of the inner vertices. They come from the centres and vertices from the icosahedra in Figure~\ref{fig:gowing3d_icosa}. 

\begin{figure}[h!]
 \centering
 \includegraphics[scale=0.85]{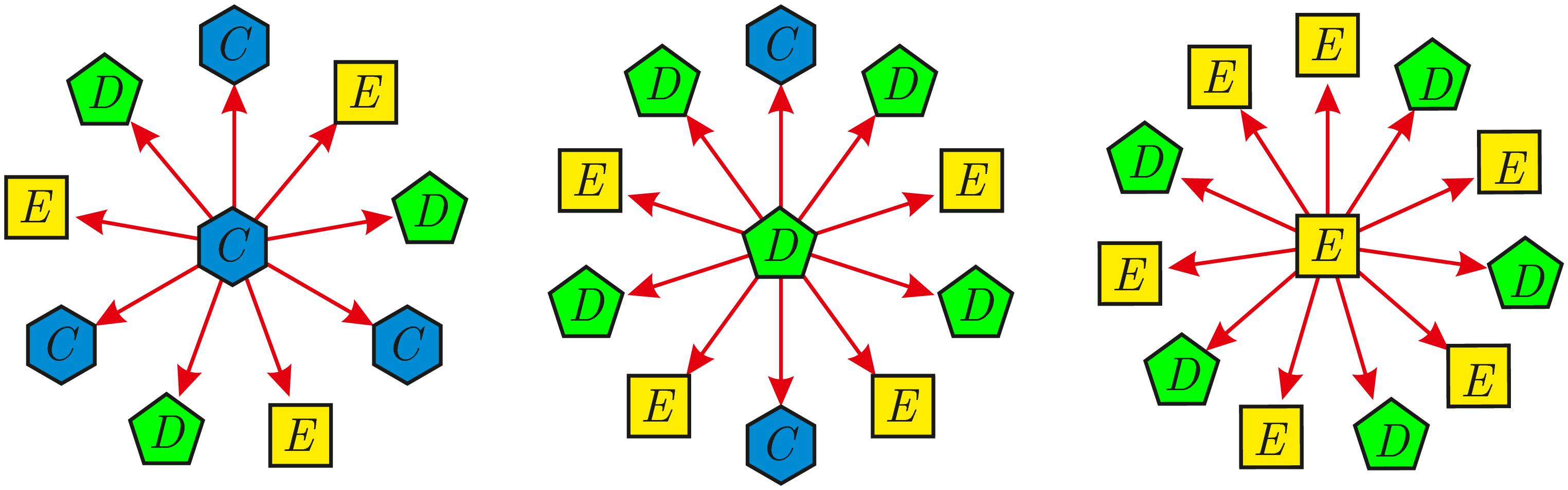}
 \caption{Growing method in case of the inner vertices}
 \label{fig:gowing3d_2}
\end{figure}

Three new $C$, $D$ and $E$ connect for all  vertices type $C$, but from the explanation above all new vertices $C$ and $D$ connect altogether three or two other vertices on level $i$, respectively. So, for the correct calculation we correspond just one third or half of them to the examined vertices $C$, respectively. All the new vertices $C$ connect to just one vertex on level $i$. By the help of the similar consideration in case of vertices $D$ and $E$, we can calculate the number of  vertices on level $i+1$, recursively, without multiplicity.

Finally, we denote the sums of  vertices types $A$, $B$, $C$,  $D$ and  $E$ on level $i$  by   $a_i$, $b_i$, $c_i$, $d_i$ and $e_i$, respectively. 

Summarising the details $(i\geq4)$ and calculating the numbers of vertices in some lower levels $(i<4)$ from Table \ref{table:typeof_vertices}, we prove the Theorem~\ref{th:growing_type}.

\begin{theorem}\label{th:growing_type}
The growing of the numbers of the different types of the vertices are described by the system of linear inhomogeneous recurrence sequences  $(n\geq1)$
\begin{equation}\label{eq:seq01}
  \begin{array}{ccl}
a_{n+1}&=& a_n+b_n+3,\\
b_{n+1}&=& a_n+2b_n,\\
c_{n+1}&=& \frac13 a_n+c_n+\frac23 d_n,\\
d_{n+1}&=& \frac12 b_n+\frac32c_n+2d_n+\frac52e_n,\\
e_{n+1}&=& 3 c_n+4d_n+6 e_n,
  \end{array}
\end{equation}
with zero initial values.
\end{theorem}

Moreover, let $s_n$ be the number of all the vertices on level $n$, so that $s_0=1$ and
\begin{equation*} \label{eq:sn}
s_{n}= a_n+b_n+c_n+d_n+ e_n+3 \qquad (n\geq1).
\end{equation*}

Table \ref{table:typeof_vertices} shows the numbers of the vertices on levels up to 10.  

\begin{table}[!hb]
  \centering \setlength{\tabcolsep}{0.4em}
\begin{tabular}{|c||c|c|c|c|c|c|c|c|c|c|c|}
  \hline
 $n$   &  0  &  1 & 2 & 3 & 4 & 5 & 6  & 7  & 8   & 9   & 10  \\ 
 \hline \hline
 $a_n$ &  0  &  0 & 3 & 6 & 12 & 27 & 66 & 168 & 435 & 1134 & 2964  \\ \hline
 $b_n$ &  0  &  0 & 0 & 3 & 12 & 36 & 99 & 264 & 696 & 1827 & 4788  \\ \hline
 $c_n$ &  0  &  0 & 0 & 1 & 3 & 9 & 34  & 174  & 1128   & 8251   & 63315         \\ \hline
 $d_n$ &  0  &  0 & 0 & 0 & 3 & 24 & 177  & 1347  & 10467   & 82029    & 644808         \\ \hline
 $e_n$ &  0  &  0 & 0 & 0 & 3 & 39 & 357  & 2952  &  23622 & 186984   & 1474773       \\ \hline
 $s_n$ &  1  &  3 & 6 & 13 & 36 & 138 & 736  & 4908  &  36351 & 280228   & 2190651       \\ \hline
 \end{tabular}
\caption{\emph{Number of types of vertices $(n\leq10)$}\label{table:typeof_vertices}}
\end{table}

\begin{theorem}\label{theorem:numvertex4q}  
The sequences $\{a_n\}$, $\{b_n\}$, $\{c_n\}$, $\{d_n\}$, $\{e_n\}$  and $\{s_n\}$ can be described by the same fifth order linear homogeneous recurrence sequence 
\begin{equation}\label{recurcde}
x_n=12x_{n-1}-37x_{n-2}+37x_{n-3}-12x_{n-4}+x_{n-5} \qquad (n\ge6),
\end{equation}
the initial values are in Table \ref{table:typeof_vertices}.
The sequences $\{a_n\}$, $\{b_n\}$ can be also described by  
\begin{equation}
x_n=4x_{n-1}-4x_{n-2}+x_{n-3} \qquad (n\ge4).\label{recurab}
\end{equation}
\noindent Moreover, the explicit formulae
\begin{eqnarray*}
a_n&=&\left(-\frac{9}{2}+\frac{21}{10}\sqrt{5}\right)\alpha_1^n+\left(-\frac{9}{2}-\frac{21}{10}\sqrt{5}\right)\alpha_2^n+3, \\
b_n&=&\left(3-\frac{6}{5}\sqrt{5}\right)\alpha_1^n+\left(3+\frac{6}{5}\sqrt{5}\right)\alpha_2^n-3,\\
c_n&=&\left(-\frac{33}{10}+\frac{3}{2}\sqrt{5}\right)\alpha_1^n+
\left(-\frac{33}{10}-\frac{3}{2}\sqrt{5}\right)\alpha_2^n+ \left(\frac{122}{15}-\frac{21}{10}\sqrt{15}\right)\alpha_3^n\\
& & \qquad 
+\left(\frac{122}{15}+\frac{21}{10}\sqrt{15}\right)\alpha_4^n+\frac13, \\
d_n&=&\left(\frac{27}{5}-\frac{12}{5}\sqrt{5}\right)\alpha_1^n+
\left(\frac{27}{5}+\frac{12}{5}\sqrt{5}\right)\alpha_2^n+ \left(-\frac{213}{20}+\frac{11}{4}\sqrt{15}\right)\alpha_3^n\\
& & \qquad+
\left(-\frac{213}{20}-\frac{11}{4}\sqrt{15}\right)\alpha_4^n-\frac32, \\
e_n&=&\left(-\frac{21}{10}+\frac{9}{10}\sqrt{5}\right)\alpha_1^n+
\left(-\frac{21}{10}-\frac{9}{10}\sqrt{5}\right)\alpha_2^n+ \left(\frac{31}{10}-\frac{4}{5}\sqrt{15}\right)\alpha_3^n\\
& & \qquad+
\left(\frac{31}{10}+\frac{4}{5}\sqrt{15}\right)\alpha_4^n+1, \\
s_n&=&\left(-\frac{3}{2}+\frac{9}{10}\sqrt{5}\right)\alpha_1^n+
\left(-\frac{3}{2}-\frac{9}{10}\sqrt{5}\right)\alpha_2^n+ \left(\frac{7}{12}-\frac{3}{20}\sqrt{15}\right)\alpha_3^n\\
& & \qquad+
\left(\frac{7}{12}+\frac{3}{20}\sqrt{15}\right)\alpha_4^n+\frac{17}6 \\
\end{eqnarray*}
are valid, where $\alpha_1=(3+\sqrt{5})/2$, $\alpha_2=(3-\sqrt{5})/2$, $\alpha_3=4+\sqrt{15}$ and $\alpha_4=4-\sqrt{15}$.
\end{theorem}

For the proof of Theorem \ref{theorem:numvertex4q} we apply Theorem \ref{th:recur}.

\begin{theorem}\label{th:recur}
Let the real linear homogeneous recurrence sequences $a^{(j)}$ embedded in each other be given the following way ($k\geq 2,
\enskip i\geq 0,\enskip  j=1,2,\dots ,k)$  
\begin{eqnarray*}
        a^{(1)}_{i+1}&=&m_{1,1}a^{(1)}_i+m_{1,2}a^{(2)}_i+ \dots +m_{1,k}a^{(k)}_i \nonumber \\
        a^{(2)}_{i+1}&=&m_{2,1}a^{(1)}_i+m_{2,2}a^{(2)}_i+ \dots +m_{2,k}a^{(k)}_i \nonumber \\
        &\vdots & \\
        a^{(k)}_{i+1}&=&m_{k,1}a^{(1)}_i+m_{k,2}a^{(2)}_i+ \dots +m_{k,k}a^{(k)}_i, \nonumber \\
 \noalign{\noindent with initial values $a^{(j)}_0 \in \mathbb{R}$ and }   \nonumber \\
        r_{i+1}&=&{\alpha}_1a^{(1)}_i+{\alpha}_2a^{(2)}_i+ \dots +{\alpha}_{k} a^{(k)}_i, \qquad r_0 \in \mathbb{R}.
\end{eqnarray*}
   In a shorter form
\begin{eqnarray}
 \vec{a}_{i+1}&=&\vec{M}\vec{a}_i,\label{eq:am}\\
    r_{i+1}&=&\veca^T\vec{a}_i, \label{eq:ra}
\end{eqnarray}
where $\vec{M}=\{m_{i,j}\}_{k\times k}$, $\vec{a}_j=[a^{(1)}_j\enskip a^{(2)}_j \enskip \dots \enskip a^{(k)}_j]^T$, $ \veca=[\alpha _1 \enskip \alpha _2 \enskip \dots \enskip \alpha _k]^T$ and $rank(\vec{M})=k$.\\
\noindent If ${\beta}_i \in \mathbb{R}$ and
\begin{eqnarray}
      r_i={\beta}_1r_{i-1}+{\beta}_2r_{i-2}+ \dots +{\beta}_kr_{i-k}, \qquad (i\geq k)\label{eq:rr}
\end{eqnarray}
then ${\beta}_i$ are the coefficients of  characteristic polynomial of  matrix $\vec{M}$.

\noindent Moreover, if
the matrix $\vec{M}$ has $\ell$ distinct eigenvalues with one algebraic multiplicity and ${\gamma}_j$ $(1\leq j\leq \ell)$ are the coefficients of  minimal polynomial of $\vec{M}$, then
\begin{eqnarray}
      r_i={\gamma}_1r_{i-1}+{\gamma}_2r_{i-2}+ \dots +{\gamma}_{\ell}r_{i-\ell} \qquad (i\geq k \geq \ell).\label{eq:rrg}
\end{eqnarray}
\end{theorem}

\begin{proof}
From \eqref{eq:am} and \eqref{eq:ra} we obtain, that
  \begin{eqnarray*}
      r_i=\veca ^T\vec{a}_{i-1}=\veca ^T\vec{M}\vec{a}_{i-2}&=&\veca ^T\vec{M}^2\vec{a}_{i-3}=\dots =\veca ^T\vec{M}^{k-1}\vec{a}_{i-k},\\
      r_{i-j}&=&\veca ^T\vec{M}^{k-(j+1)}\vec{a}_{i-k}  \qquad (j=0,1,\dots ,k-1).\\[2mm]
 \noalign{\noindent It follows from $rank(\vec{M})=k$ (has inverse), that}\\[-3.5mm]
      r_{i-k}&=&\veca ^T\vec{M}^{-1}\vec{a}_{i-k}.
  \end{eqnarray*}
We substitute the results into (\ref{eq:rr}),
 \begin{eqnarray*}
   \veca ^T\vec{M}^{k-1}\vec{a}_{i-k}&=&{\beta}_1 \veca ^T\vec{M}^{k-2}\vec{a}_{i-k}+ \dots + {\beta}_j \veca ^T\vec{M}^{k-(j+1)}\vec{a}_{i-k}+\dots
    +{\beta}_k \alpha ^T\vec{M}^{-1}\vec{a}_{i-k}  \\
   &=&\sum _{j=1}^k \Big( {\beta}_j \veca ^T\vec{M}^{k-(j+1)}\vec{a}_{i-k}\Big)=\veca ^T \Bigg( \sum _{j=1}^k \Big(
          {\beta}_j \vec{M}^{k-(j+1)}\Big) \Bigg)\vec{a}_{i-k} .
\end{eqnarray*}
We gain, that
\begin{eqnarray*}
   \veca ^T \Bigg(\vec{M}^{k-1}- \sum _{j=1}^k \Big(
          {\beta}_j \vec{M}^{k-(j+1)}\Big) \Bigg)  \vec{a}_{i-k}= 0 .
\end{eqnarray*}
As $\veca ^T$  and $\vec{a}_{i-k}$ can be any elements of the vector space $\mathbb{R}^k$,
\begin{eqnarray*}
   \vec{M}^{k-1}- \sum _{j=1}^k \Big(
          {\beta}_j \vec{M}^{k-(j+1)}\Big)  =   \vec{0},
\end{eqnarray*}
 thus
\begin{eqnarray}\label{eq:charM}
      \vec{M}^{k}&=&\sum _{j=1}^k  {\beta}_j \vec{M}^{k-j}.
\end{eqnarray}
Using the well-known  \emph{Cayley-Hamilton Theorem}, from \eqref{eq:charM} the equation
\begin{equation*}
x^k={\beta}_1x^{k-1}+{\beta}_2x^{k-2}+ \dots +{\beta}_k
\end{equation*}
is the characteristic equation of matrix $\vec{M}$. 
If the matrix $\vec{M}$ has $\ell$ distinct eigenvalues with one algebraic multiplicity and ${\gamma}_j$  are the coefficients of the minimal polynomial of $\vec{M}$, then the method of the proof can be followed step by step for $\ell$ elements of $r_i$ and for ${\gamma}_j$ coefficients too, thus \eqref{eq:rrg} also holds.
\end{proof}

\begin{proof}[Proof of Theorem \ref{theorem:numvertex4q}]
Let $v_n=3$ $(n\geq1)$ be a constant sequence and $v_0=1$. The value $v_n$  gives the number of  vertices type ``1" on  level $n$. Substitute  $3=v_n$ into the first equation of \eqref{eq:seq01} and complete the equations system \eqref{eq:seq01} with $v_{n+1}=v_n$. 
Than we have the system of linear homogeneous recurrence sequences $(n\geq1)$
\begin{equation*}\label{eq:seq01v}
  \begin{array}{ccl}
a_{n+1}&=& a_n+b_n+ v_n,\\
b_{n+1}&=& a_n+2b_n,\\
c_{n+1}&=& \frac13 a_n+c_n+\frac23 d_n,\\
d_{n+1}&=& \frac12 b_n+\frac32c_n+2d_n+\frac52e_n,\\
e_{n+1}&=& 3 c_n+4d_n+6 e_n,\\
v_{n+1}&=&  v_n
  \end{array}
\end{equation*}
and  
\begin{equation*}\label{eq:snv}
s_{n}= a_n+b_n+c_n+d_n+ e_n+v_n \qquad (n\geq0).
\end{equation*}

Using the results of Theorem \ref{th:recur} when   
$$\vec{M}=\begin{pmatrix}
1   & 1   & 0   & 0   & 0   & 1 \\ 
1   & 2   & 0   & 0   & 0   & 0 \\ 
\frac13 & 0   & 1   & \frac23 & 0   & 0 \\ 
0   & \frac12 & \frac32 & 2   & \frac52 & 0 \\ 
0   & 0   & 3   & 4   & 6   & 0 \\ 
0   & 0   & 0   & 0   & 0   & 1 \\ 
\end{pmatrix},$$
$\vec{a}_j=[a_j\enskip b_j\enskip c_j\enskip e_j\enskip v_j\enskip]^T$,  and $rank(\vec{M})=6$ 
we gain that the solutions of system of linear recurrence equations \eqref{eq:charM} are $\beta_1=12-t$, $\beta_2=-37+12t$, $\beta_3=37-37t$,
$\beta_4=-12+37t$, $\beta_5=1-12t$, $\beta_6=t$, where $t\in \mathbb{R}$. As $r_n$ was an arbitrary equation, $r_n$ can be $s_n$, $a_n$, \ldots\ $e_n$ with $\veca=(1,1,1,1,1,1)$, $\veca=(1,0,0,0,0,0)$, \ldots, $\veca=(0,0,0,0,1,0)$, respectively.  Moreover, let $t=0$, then we obtain the (degenerate) recurrence sequence \eqref{recurcde}.

As $a_{n+1}$, $b_{n+1}$ and $v_{n+1}$ are independent from $c_n$, $d_n$ and $e_n$, they form a system of homogeneous recurrence equations again with matrix 
$\vec{M}_{ab}=\left( \begin{smallmatrix}
1   & 1   & 1    \\ 
1   & 2   & 0    \\ 
0   & 0   & 1    \\ 
\end{smallmatrix}\right)$. 
Using the results of  Theorem \ref{th:recur} again we gain $\beta_1=4$, $\beta_2=-4$, $\beta_3=1$, so the equation \eqref{recurab} holds.

The characteristic equation of \eqref{recurcde} is
\begin{equation}\label{eq:mini}
x^5={12}x^{4}-37x^{3}+37x^2-12x +1
\end{equation} 
and its solutions are $\alpha_1=(3+\sqrt{5})/2$, $\alpha_2=(3-\sqrt{5})/2$, $\alpha_3=4+\sqrt{15}$ and $\alpha_4=4-\sqrt{15}$ and $\alpha_5=1$.  We mention that equation \eqref{eq:mini} is the minimal polynomial of the matrix $\vec{M}$. (The roots of the  characteristic equation of \eqref{recurab}, $x^3=4x^2-4x +1$  are also $\alpha_1$, $\alpha_2$ and $\alpha_5$.) A suitable linear combination of their $n^{\text{th}}$ power provide the explicit formulae (\cite{sho}).     
\end{proof}

\begin{rem}
In Pascal's pyramid the equations system \eqref{eq:seq01} also holds with suitable initial values. In this case, there is no type vertices $B$, $D$ and $E$, so $b_i=d_i=e_i=0$ for any $i$. Thus the hyperbolic Pascal pyramid is not only the geometric but also the algebraic generalization of Pascal's pyramid.   
\end{rem}

\begin{rem}
The ratios of numbers of vertices from level to level tend to the biggest eigenvalue of the matrix $\vec{M}$. So, the growing ratio of \hpp is $\alpha_3=4+\sqrt{15}\approx 7.873$, on the contrary it is $1$ in case of the Euclidean case.    
\end{rem}

\section{Sum of the values on levels in the hyperbolic Pascal pyramid}

The sum of the values of the elements on  level $n$ in the classical Pascal's pyramid is $3^n$ (\cite{B}). In this section we determine it in case of the hyperbolic Pascal pyramid. 

Denote respectively $\hat{a}_{n}$, $\hat{b}_{n}$, $\hat{c}_{n}$, $\hat{d}_{n}$ and $\hat{e}_{n}$ the sums of the values of  vertices type $A$, $B$, $C$, $D$ and $E$ on level $n$, and let $\hat{s}_{n}$ be the sum of all the values. 
From Figures \ref{fig:gowing3d_1} and \ref{fig:gowing3d_2} the results of Theorem \ref{th:recursum} can be read directly. For example for all  vertices type $A$, $B$ and $1$ on level $i$ generate two vertices type $A$ on level $i+1$ and it follows the first equation of \eqref{eq:seq02}.
Table \ref{table:sumof_vertices} shows the sum of the values of the vertices on levels up to 10.  

\begin{theorem} \label{th:recursum}
If $n\geq1$, then
\begin{equation}\label{eq:seq02}
  \begin{array}{ccl}
\hat{a}_{n+1}&=& 2\hat{a}_n+2\hat{b}_n+6,\\
\hat{b}_{n+1}&=& \hat{a}_n+2\hat{b}_n,\\
\hat{c}_{n+1}&=&  \hat{a}_n+3\hat{c}_n+2 d_n,\\
\hat{d}_{n+1}&=&  \hat{b}_n+3c_n+4\hat{d}_n+5\hat{e}_n,\\
\hat{e}_{n+1}&=& 3 \hat{c}_n+4\hat{d}_n+6 \hat{e}_n
  \end{array}
\end{equation}
with zero initial values.
\end{theorem}

Table \ref{table:sumof_vertices} shows the sums of  values on  levels up to 10.

\begin{table}[!htb]
  \centering \setlength{\tabcolsep}{0.4em}
\begin{tabular}{|c||c|c|c|c|c|c|c|c|c|c|c|}
  \hline
 $n$   &  0  &  1 & 2 & 3 & 4 & 5 & 6  & 7  & 8   & 9   & 10  \\ 
 \hline \hline
 $a_n$ &  0  &  0 & 6 & 18 & 54 & 174 & 582 & 1974 & 6726 & 22950 & 78342  \\ \hline
 $b_n$ &  0  &  0 & 0 & 6 & 30 & 114 & 402 & 1386 & 4746 & 16218 & 55386  \\ \hline
 $c_n$ &  0  &  0 & 0 & 6 & 36 & 210 & 1452  & 12138  & 114684   & 1147002   & 11729148         \\ \hline
 $d_n$ &  0  &  0 & 0 & 0 & 24 & 324 & 3600  & 38148  & 398112   & 4132596    & 42818208        \\ \hline
 $e_n$ &  0  &  0 & 0 & 0 & 18 & 312 & 3798 & 41544 & 438270  &  4566120 & 47368110        \\ \hline
 $s_n$ &  1  &  3 & 9 & 33 & 165 & 1137 & 9837 & 95193  & 962541  &  9884889 & 102049197         \\ \hline
 \end{tabular}
\caption{\emph{Sum of  values of  vertices $(n\leq10)$}\label{table:sumof_vertices}}
\end{table}

Let $\hat{s}_n$ be the sum of the values of all the vertices on level $n$, then $\hat{s}_0=1$ and 
\begin{equation*}
\hat{s}_n = \hat{a}_n+\hat{b}_n+\hat{c}_n+\hat{d}_n+\hat{e}_n+3 \qquad (n\geq1).
\end{equation*}

\begin{theorem}
The sequences $\{\hat{a}_n\}$, $\{\hat{b}_n\}$, $\{\hat{c}_n\}$, $\{\hat{d}_n\}$, $\{\hat{e}_n\}$  and $\{\hat{s}_n\}$ can be described by the same sixth order linear homogeneous recurrence sequence 
\begin{equation*}\label{recurcdehat}
\hat{x}_n = 18 \hat{x}_{n-1}-99\hat{x}_{n-2}+226\hat{x}_{n-3}-224\hat{x}_{n-4}+92\hat{x}_{n-5}-12\hat{x}_{n-6} \qquad (n\ge7),
\end{equation*}
the initial values are in Table \ref{table:sumof_vertices}.
The sequences $\{\hat{a}_n\}$, $\{\hat{b}_n\}$ can be also described by  
\begin{equation*}\label{recurabhat}
\hat{x}_n = 5 \hat{x}_{n-1}-6\hat{x}_{n-2}+2\hat{x}_{n-3} \qquad (n\ge4).
\end{equation*}
\noindent The explicit formulae
\begin{eqnarray*}
\hat{a}_n &=&\left(\frac92\,\sqrt {2}-6 \right)  \left( 2+\sqrt {2} \right) ^{n}+ \left(\frac92\,\sqrt {2}-6 \right)  \left( 2-\sqrt {2} \right) ^{n}+6,\\
\hat{b}_n&=& \left(\frac92 -3\,\sqrt {2} \right)  \left( 2+\sqrt {2} \right)^{n} + \left( \frac92+3\,\sqrt {2} \right)  \left( 2-\sqrt {2} \right) ^{n}  -6 
\end{eqnarray*}
and
\begin{multline*}
\hat{s}_n=3+ \frac32\left(\sqrt {2}-1 \right)  \left( 2+\sqrt {2} \right) ^{n}-
 \frac32\left(\sqrt {2}+1 \right) \left( 2-\sqrt {2} \right)^{n} +\delta_4\alpha_4+\delta_5\alpha_5+\delta_6\alpha_6,
\end{multline*}
where if $\varphi=\arctan(9\sqrt{101}/128)/3$, then
$\alpha_4=(-\sqrt{85}\cos(\varphi)-\sqrt{3}\sqrt{85}\sin(\varphi)+13)/3 \approx 0.240683$,
$\alpha_5=(-\sqrt{85}\cos(\varphi)+\sqrt{3}\sqrt{85}\sin(\varphi)+13)/3 \approx 2.408387$ and 
$\alpha_6=(2\sqrt{85}\cos(\varphi)+13)/3 \approx 10.350930$
are the roots of the equation ${x}^{3}-13{x}^{2}+28x-6=0$, moreover $\delta_4\approx 1.137480 $, $\delta_5\approx -0.144699$ and $\delta_6\approx 0.007219$.
\end{theorem}

\begin{proof}
Follow the proof of  Theorem \ref{theorem:numvertex4q} step by step.
Using the results of  Theorem \ref{th:recursum} when   
$$\vec{M}=\begin{pmatrix}
2   & 2   & 0   & 0   & 0   & 2 \\ 
1   & 2   & 0   & 0   & 0   & 0 \\ 
1   & 0   & 3   & 2   & 0   & 0 \\ 
0   & 1   & 3   & 4   & 5   & 0 \\ 
0   & 0   & 3   & 4   & 6   & 0 \\ 
0   & 0   & 0   & 0   & 0   & 1 \\ 
\end{pmatrix},\qquad
\vec{M}_{ab}=\left( \begin{matrix}
2   & 2   & 2    \\ 
1   & 2   & 0    \\ 
0   & 0   & 1    \\ 
\end{matrix}\right),$$
we gain that $\beta_1 = 18$, $\beta_2 = -99$, $\beta_3 = 226$, $\beta_4 = -224$, $\beta_5 = 92$, $\beta_6 = -12$ and the characteristic polynomial of $\vec{M}$ is $\left( x-1 \right)  \left( {x}^{2}-4\,x+2 \right)  \left( {x}^{3}-13\,{x}^{2}+28\,x-6 \right)$ and its roots are 
$\alpha_i$ $(i=1, \dots , 6)$. As the exact values of coefficients $\delta_j$ 
 $(j=4,\, 5,\, 6)$ are very complicated (for sequences $\hat{c}_n$, $\hat{d}_n$, $\hat{e}_n$ also), we give their numerical values in case sequence $\hat{s}_n$ by the help of MAPLE software. But from the characteristics polynomial of $\vec{M}_{ab}$ the sequences $\hat{a}_n$ and $\hat{b}_n$  are given in exact explicit forms. 
\end{proof}

\begin{rem}
The growing ratio of  values of \hpp is $\approx\!10.351$, while it is $3$ in case of the Euclidean case.    
\end{rem}

\end{document}